\documentclass{article}
\usepackage{amsthm}
\newtheorem{theorem}{Theorem}
\newtheorem{corollary}{Corollary}
\newtheorem{proposition}{Proposition}
\newtheorem{lemma}{Lemma}
\newtheorem{remark}{Remark}
\usepackage{fullpage}

\usepackage{amssymb,amsmath}

\usepackage{hyperref}

\usepackage{graphics}
\usepackage{epstopdf}

\usepackage{algorithm}
\usepackage{algorithmic}
\usepackage{epstopdf}
\usepackage{graphicx}
\usepackage{multirow}
\usepackage{color}
\usepackage[colorinlistoftodos,bordercolor=orange,backgroundcolor=orange!20,linecolor=orange,textsize=scriptsize]{todonotes}

\usepackage{pdflscape}
\usepackage{mathtools}

\DeclarePairedDelimiter\ceil{\lceil}{\rceil}

\newcommand{\norm}[1]{\left \lVert #1 \right\rVert}
\def\<#1,#2>{\langle #1,#2\rangle}

\newcommand{\ve}[2]{\langle #1 ,  #2 \rangle}   
\newcommand{\eqdef}{\stackrel{\text{def}}{=}}
\newcommand{\R}{\mathbb{R}}
\newcommand{\cA}{\mathcal{A}}

\newcommand{\Exp}{\mathbf{E}}
\newcommand{\Prob}{\mathbf{P}}

\newcommand{\bN}{\mathbb{N}}
\newcommand{\cX}{\mathcal{X}}

\newcommand{\dom}{\operatorname{dom}}
\newcommand{\dist}{\operatorname{dist}}

\newtheorem{ass}{Assumption}

\author{Olivier Fercoq\thanks{Olivier Fercoq, LTCI,
	    T\'el\'ecom ParisTech, Universit\'e Paris-Saclay, Paris,  France,
              \texttt{olivier.fercoq@telecom-paristech.fr}} \and Zheng Qu\thanks{Zheng Qu, Department of Mathematics, The University of Hong Kong,
              Hong Kong, China,   \texttt{zhengqu@maths.hku.hk}}}

\newcommand{\TheTitle}{Restarting the accelerated coordinate descent method with a rough strong convexity estimate} 
 
\title{\TheTitle\thanks{The first author's work was supported by the
EPSRC  Grant EP/K02325X/1
{\em Accelerated Coordinate Descent Methods for Big Data Optimization},
the Centre for Numerical Algorithms and Intelligent Software (funded by EPSRC grant EP/G036136/1 and the Scottish Funding Council), the Orange/Telecom ParisTech think tank Phi-TAB, the ANR grant ANR-11-LABX-0056-LMH, LabEx LMH as part of the Investissement d'avenir project. 
The second author's work was supported by Hong Kong Research Grant Council 27302016. This research was partially conducted using the HKU Information Technology Services research computing facilities that are supported in part by the Hong Kong UGC Special Equipment Grant (SEG HKU09).}}

\begin{document}

\maketitle

\begin{abstract}
We propose new restarting strategies for the accelerated coordinate descent method.
Our main contribution is to show that for a well chosen sequence of  restarting times, the restarted method has a nearly geometric rate
of convergence. A major feature of the method is that it can take profit of the local quadratic error bound of the objective function without knowing the actual
value of the error bound. 
We also show that under the more restrictive assumption that the objective function is strongly convex, any fixed restart period leads to a geometric rate of convergence.
Finally, we illustrate the properties of the algorithm on a regularized logistic regression problem and on a Lasso problem.

\end{abstract}

\section{Introduction}

\subsection{Motivation}

We consider in this paper the minimization of composite
convex functions of the form 
$$F(x) = f(x) + \psi(x),\enspace x\in \R^n$$
where $f$ is differentiable with Lipschitz gradient 
and $\psi$ may be nonsmooth but is separable, and has an easily computable proximal operator.  
Coordinate descent methods are often considered in this context thanks to the separability of the proximal operator of~$\psi$.
These are optimization algorithms that update only one coordinate of the
 vector of variables at each iteration, hence using partial 
 derivatives rather than the whole gradient. 

Similarly to what he had done for the gradient method, 
Nesterov introduced, for smooth functions, the randomized accelerated coordinate descent method with
an improved guarantee on the iteration complexity~\cite{Nesterov:2010RCDM}.
Indeed, for a mild additional computational cost, accelerated methods transform 
the proximal coordinate descent method, for which the optimality gap 
$F(x_k) - F^*$ decreases as $O(1/k)$, into an algorithm with 
``optimal'' $O(1/k^2)$ complexity~\cite{nesterov1983method}.  
\cite{lee2013efficient} introduced an efficient implementation of the method 
and 
\cite{FR:2013approx} developed the accelerated parallel and proximal coordinate descent method (APPROX) for the minimization of composite functions.

When solving a strongly convex problem, the classical
(non-accelerated) gradient and coordinate descent methods
automatically have a linear rate of convergence,
i.e. $F(x_k) - F^* \in O((1-\mu)^k)$
for a problem dependent $0<\mu<1$, whereas one needs to know explicitly the strong convexity parameter in order
to set accelerated gradient and accelerated coordinate descent methods to have a linear rate of convergence, 
see for instance~\cite{lee2013efficient,UniversalCatalyst,lin2014accelerated,Nesterov:2010RCDM,nesterov2013gradient}. 
Setting the algorithm with an incorrect parameter may result in a slower algorithm, sometimes even slower than if
we had not tried to set an acceleration scheme~\cite{o2012adaptive}.
This is a major drawback of the method because in general, the strong convexity parameter is
difficult to estimate. 

In the context of accelerated gradient method with unknown strong convexity parameter, Nesterov~\cite{nesterov2013gradient} proposed a restarting scheme which adaptively approximate the strong convexity parameter. The same idea was exploited in~\cite{Lin2015} for sparse optimization.
\cite{nesterov2013gradient} also showed that,
instead of deriving a new method designed to work better for strongly convex functions,
one can restart the accelerated gradient method and get a linear convergence rate. 
It was later shown in~\cite{LiuYang17,FercoqQuDeter17} that 
a local quadratic error bound is sufficient to get a global linear rate of convergence.

The adaptive restart of randomized accelerated coordinate descent methods is more complex than in the 
deterministic case. As the complexity bound holds in expectation only, one cannot rely 
on this bound to estimate whether the rate of convergence is in line with our estimate of the
local error bound, as was done in the deterministic case. Instead, \cite{fercoqqu2016restarting}
proposed a fixed restarting scheme. They needed to restart at a point which is a convex
combination of all previous iterates and required stronger assumptions than in the present work.

\subsection{Contributions}
 
In this paper, we show how restarting the accelerated coordinate
descent method can help us take profit of the local quadratic
error bound of the objective, when this property holds. 

We consider three setups: 
\begin{enumerate}
\item If the local quadratic
error bound coefficient $\mu$ of the objective function is known, then we show that setting
a restarting period as $O(1/\sqrt{\mu})$ yields an algorithm
with optimal rate of convergence. More precisely restarted APPROX admits the same theoretical complexity bound as the accelerated coordinate descent methods for strongly convex functions developed in~\cite{lin2014accelerated}, is applicable with milder assumptions and exhibits better performance in numerical experiments.

\item If the objective function is strongly convex, we
show that we can restart the accelerated coordinate descent method 
at {\em any} frequency and get a linearly convergent algorithm. 
The rate depends on an estimate of the local quadratic error bound 
and we show that for a wide range of this parameter,
one obtains a faster rate than without acceleration.
In particular, we do not require the estimate
of the error bound coefficient to be smaller than
the actual value. The difference with respect to \cite{fercoqqu2016restarting}
is that in this section, we show that there is no need to restart
at a complex combination of previous iterates.

\item If the local error bound coefficient is not known, 
we introduce a variable restarting periods and show that up to a $\log(\log 1/\epsilon)$ term, the algorithm
is as efficient as if we had known the local error bound coefficient.
\end{enumerate}

In Section~\ref{sec:basic-approx} we recall
the main convergence results for the accelerated proximal coordinate descent method (APPROX) and present restarted APPROX.
In Section~\ref{sec:restart}, we study restart for APPROX with a fixed restart period and in Section~\ref{sec:variable_restart} we give the algorithm with variable restarting periods. Finally, we present numerical experiments
on the lasso and logistic regression problem in Section~\ref{sec:expe}.


\section{Problems, assumptions and algorithms}
\label{sec:basic-approx}

In this section, we present in detail the problem we are studying. We also 
recall basic facts about the accelerated coordinate descent method that will 
be useful in the analysis.

\subsection{Problem and assumptions}
\label{sec:problem}

For simplicity we present the algorithm in  coordinatewise form. The extension to blockwise setting follows naturally (see for instance~\cite{FR:2013approx}).
We consider the following optimization problem:
\begin{equation}\label{eq-prob}
\begin{array}{ll}
 \mathrm{minimize} & F(x):=f(x)+\psi(x)\\
\mathrm{subject~to~} & x=(x^1,\dots,x^n)\in \R^n,
\end{array}
\end{equation}
where $f:\R^n\rightarrow \R$ is a differentiable convex function and $\psi:\R^n\rightarrow \R\cup\{+\infty\}$ is a closed convex  and separable function:
$$
\psi(x)=\sum_{i=1}^n\psi^i(x^i).
$$
Note that this implies that each function $\psi^i:\R\rightarrow \R$ is closed and convex. 
We denote by $F^*$ the optimal value of~\eqref{eq-prob} and assume that the optimal solution set $\cX^*$ is nonempty.  For any positive vector $v\in \R^n_+$, we denote by $\|\cdot\|_v$ the weighted  Euclidean norm:
 \[\|x\|_v^2 \eqdef \sum_{i=1}^n v_i (x^i)^2,\]
and $\dist_v(x, \cX^*)$ the distance of $x$ to the set $\cX^*$ with respect to the norm
$\|\cdot\|_v$.

Throughout the paper, we will assume that the objective function satisfies the
following local error bound condition.
\begin{ass}\label{ass2}
For any $x_0\in \dom(F) $, there is $\mu>0$ such that 
\begin{align}\label{a:strconv2}
F(x)\geq F^* +\frac{\mu}{2}\dist_v(x, \cX^*)^2, \enspace \forall x\in [F\leq F(x_0)],
\end{align}
where $[F\leq  F(x_0)]$ denotes the set of all $x$ such that $F(x)\leq F(x_0)$.
\end{ass}
We denote by $\mu(v,x_0)$ the largest $\mu>0$ satisfying~\eqref{a:strconv2} and by $\mu(x_0)$ the value of $\mu(\vec{e},x_0)$ where $\vec{e}$ is the unit vector. Note that   
\begin{align}\label{a:muv}\min_i \mu(x_0)/v_i \leq \mu(v,x_0)\leq\max_i \mu(x_0)/v_i.
\end{align}

The fact that we assume that~\eqref{a:strconv2} holds for any $x_0 \in \dom(F)$ is not restrictive. The proposition below shows that, if it holds for a given $x_0$, it will hold on any bounded set.
\begin{proposition}
\label{prop:localerrorboundonallcompacts}
If $F$ is convex and satisfies the local error bound Assumption~\ref{ass2}
then for all $M \geq 1$, if $F(x') - F^* = M (F(x_0) - F^*)$, then $\mu(v, x') \geq \mu(v, x_0)/M$.
\end{proposition}
\begin{proof}
Suppose there exists $\mu$, $v$ and $r = F(x_0) - F^*$ such that $\forall x\in [F - F^*\leq r]$,
\begin{align*}
F(x)\geq F^* + \frac{\mu}{2}\dist_v(x, \cX^*)^2 \enspace .
\end{align*}
Let $y \in \dom(F)$ such that $r\leq  F(y) - F^* \leq rM$, $p$ be the orthogonal projection of $y$ onto $\cX^*$ and $x = p + \frac{r}{F(y) - F^*} (y-p)$.
As $F$ is convex,  $
F(x) \leq \frac{r}{F(y) - F^*} F(y) + \big(1-\frac{r}{F(y) - F^*}\big)F(p)
$. Note in particular that $F(x) - F^* \leq r$.
Rearranging, we get
\begin{align*}
F(y) &\geq \frac{F(y) - F^*}{r} F(x) - \big(\frac{F(y) - F^*}{r}-1\big)F^* 
\\&\geq F^* + \frac{F(y) - F^*}{r}\frac{\mu}{2}\dist_v(x, \cX^*)^2\\
& = F^*+\frac{r}{F(y) - F^*}\frac{\mu}{2}\dist_v(y, \cX^*)^2 \geq F^*+\frac{1}{M} \frac{\mu}{2}\dist_v(y, \cX^*)^2
\end{align*}
\end{proof}

Condition~\eqref{a:strconv2} is sometimes referred to as quadratic (functional) growth condition, as it controls the growth of the objective function value by the squared distance between any point in the sublevel set and the optimal set. This geometric property is equivalent to the {\L}ojasiewicz gradient inequality with exponent $1/2$~\cite{Bolte2017}, as well as to the so-called first-order error bound condition which bounds the distance by the norm of the proximal residual~\cite{DrusvyatskLewis}.  
A wide range of stuctured optimization models that arise in application~\cite{NecoaraNesGli} satisfy Assumption~\ref{ass2}.   
The constant $\mu(x_0)$ embodies in some sense the geometrical complexity of the problem. We recall a prototype  for which lower bounds for $\mu(x_0)$ can be deduced. It is a composition of
strongly convex function with linear term under polyhedral constraint:
\begin{align}\label{a:compositionproblem}
\min_{x} F_1(x)\equiv g(Ax)+c^\top x+1_{Cx\leq d }
\end{align}
where $A$, $C$, $c$ and $d$ are matrices with appropriate dimensions and $g$ is a smooth and strongly convex function. In this case the optimal set is polyhedral and can be written as:
\begin{align}\label{a:lis}
\cX^*=\{x: Ax=t^*, c^\top x=s^*, Cx\leq d\},
\end{align} for uniquely determined $t^*$ and $s^*$. It was shown
in~\cite{NecoaraNesGli} that
\begin{align}\label{a:estimationofmu}
\mu_{F_1}(x_0) \geq \frac{\sigma_g}{\theta^2(1+(F_1(x_0)-F_1^*)\sigma_g+2\|\nabla g(t^*)\|^2)}
\end{align}
where $\sigma_g$ is the strong convexity parameter of $g$ and $\theta$ is the Hoffman constant for the polyhedral optimal set $\cX^*$. 
 The dependence on $F_1^*$ and $\cX^*$ in the bound~\eqref{a:estimationofmu} can be 
further replaced by known constants if the constraint set $\{Cx\leq d\}$ is compact, see for example~\cite{Beck2017}.
Hence the estimation of $\mu_{F_1}(x_0)$ mainly relies on the computation of 
the Hoffman constant $\theta$ for the linear inequality system~\eqref{a:lis}. Assuming that $A$ has full row rank and 
$\{x:Ax=0,c^\top x=0, Cx<0\}\neq \emptyset$, then, we have~\cite{Klatte1995}:
\begin{align}\label{a:thetaformu}
\theta=\min\{\|A^\top u+c^\top w+C^\top v\|: \|u\|^2+\|v\|^2+\|w\|^2=1, v\geq 0\}.
\end{align}

Many problems fall into the category of~\eqref{a:compositionproblem}, including
 the $L^1$ regularized least square problem~\eqref{a:LASSO} and the logistic regression problem~\eqref{alr} that we shall consider later for numerical illustration.  Indeed the $L^1$ regularization term can be written equivalently 
as a system of linear inequalities $\{Cx\leq d\}$, as shown in~\cite{Bolte2017}. However, the size of the matrix $C$ 
shall be of the same order as the number of variables $n$\footnote{In~\cite{Bolte2017}, the $L^1$ norm is rewritten by
a matrix of $2^n$ rows. }. Therefore, the computation of $\theta$ and estimation of $\mu(x_0)$ is far from trivial when the problem dimension
is high and can sometimes be as hard as solving the original optimization problem. Note that  under Assumption~\ref{ass2},  a broad class of first order methods, including the proximal gradient method and its coordinatewise extensions~\cite{Luo1993,JMLR:v15:wang14a}, converge linearly without  requiring a priori knowledge on the error bound constant $\mu(x_0)$.

\subsection{APPROX and its properties}
In the following,  $\nabla f(y_k)$ denotes the gradient of $f$ at point $y_k$ and  $\nabla_i f(y_k)$ denotes the partial derivative of $f$ at point $y_k$ with respect to the $i$th coordinate.   $\hat S$ is  a random subset of $[n]:= \{1,2,\dots,n\}$ with the property that $\Prob(i \in \hat{S})=\Prob(j\in \hat{S})$ for all $i,j \in [n]$ and $\tau=\Exp[|\hat S|]$.

We recall the definition of APPROX(Accelerated Parallel PROXimal coordinate descent method) in Algorithm~\ref{alg:approx}. The algorithm reduces to the accelerated proximal gradient (APG) method~\cite{tseng2008accelerated} when $\hat S=[n]$ with probability one.
It employs a positive parameter vector $v\in \R^n$. To guarantee the convergence of the algorithm, the positive vector $v$ should satisfy the so-called expected separable overapproximation (ESO) assumption, developed in~\cite{FR:spcdm,RT:PCDM} for the study of parallel coordinate descent methods.
\begin{ass}[ESO] \label{ass:ESO} 
We write $(f,\hat{S})\sim \mathrm{ESO}(v)$ if
\begin{equation}\label{eq:ESO}\Exp\left[f(x+h_{[\hat{S}]})\right] \leq f(x) + \frac{\tau}{n}\left(\ve{\nabla f(x)}{h} + \frac{1}{2}\|h\|_{v}^2\right), \qquad x,h \in \R^n.
\end{equation}
where for $h=(h^1,\dots, h^n)\in \R^n$ and $S\subset [n]$, $h_{[S]}$ is defined as: \[ h_{[S]} \eqdef \sum_{i\in S} h^i e_i ,\] with $e_i$ being the $i$th standard basis vectors in $\R^n$.
\end{ass}
We require that the positive vector $v$ used in APPROX satisfy~\eqref{eq:ESO} with respect to the sampling $\hat S$ used. 
When in each step we update only one coordinate, we have $\tau=1$ and~\eqref{eq:ESO} reduces to:
\begin{equation}\label{eq:ESO-2}
\frac{1}{n}\sum_{i=1}^n  f(x+h^ie_i) \leq f(x) + \frac{1}{n}\left(\ve{\nabla f(x)}{h} + \frac{1}{2}\|h\|_{v}^2\right), \qquad x,h \in \R^n.
\end{equation}
It is easy to see that in this case the vector $v$ corresponds to the coordinate-wise Lipschitz constants of $\nabla f$, see e.g.~\cite{Nesterov:2010RCDM}.  Explicit formulas for computing admissible $v$ with respect to more general sampling $\hat S$ can be found in~\cite{RT:PCDM,FR:spcdm,QR:acdn}.
\begin{algorithm}
\begin{algorithmic}[h!]
\STATE{Set $\theta_0 = \frac{\tau}{n}$ and $z_0 = x_0$}.
\FOR{$k \in \{0, \ldots, K-1\}$} 
\STATE $y_k=(1-\theta_k)x_k+\theta_k z_k$ \\
\STATE Randomly generate $S_k \sim \hat{S}$\\
\FOR{ $i \in S_k$}
\STATE
$
z_{k+1}^i
=\arg\min_{z\in \R} \big\{\< \nabla_i f(y_k), z-y_k^i>+\frac{\theta_kn v_i}{2 \tau} |z-z_k^i|^2+\psi^i(z) \big\}
$ \\
\ENDFOR
\STATE $x_{k+1}=y_k+\frac{n}{\tau}\theta_k(z_{k+1}-z_k)$\\
\STATE $\theta_{k+1}=\frac{\sqrt{\theta_k^4 + 4 \theta_k^2} - \theta_k^2}{2}$
\ENDFOR
\RETURN $x_{K}$
\end{algorithmic}
\caption{APPROX($f, \psi, x_0, K$)~\cite{FR:2013approx}}
\label{alg:approx}
\end{algorithm}

The rest of this section recalls some basic results about APPROX that we shall need.

\begin{lemma}\label{l:thetak}
The sequence $(\theta_k)$ defined by $\theta_0 \leq 1$ and $\theta_{k+1} = \frac{\sqrt{\theta_k^4 + 4 \theta_k^2} - \theta_k^2}{2}$ satisfies
\begin{align}&\label{athetabd}\frac{(2-\theta_0)}{k+(2-\theta_0)/\theta_0} \leq \theta_k \leq \frac{2}{k+2/\theta_0}
\\\label{arectheta}
&\frac{1-\theta_{k+1}}{\theta_{k+1}^2}=\frac{1}{\theta_k^2},\enspace \forall k=0,1,\dots
\end{align}
\end{lemma}
Part of the results are proved in~\cite{FR:2013approx}. We give the complete proof in the appendix.

\begin{proposition}[\cite{FR:2013approx}]
\label{prop:approx_basic}
The iterates of APPROX (Algorithm~\ref{alg:approx}) satisfy for all $k \geq 1$   and any $x_*\in\cX^*$,
\begin{equation*}
\frac{1}{\theta_{k-1}^2}\Exp[F(x_{k}) - F^*] + \frac{1}{2\theta_0^2}\Exp[\norm{z_{k} - x_*}_v^2] \leq \frac{1 - \theta_0}{\theta_{0}^2}(F(x_{0}) - F^*) + \frac{1}{2\theta_0^2}\norm{x_{0} - x_*}_v^2 
\end{equation*}

\end{proposition}

\subsection{Restarted APPROX}

Under Assumption~\ref{ass2}, restarted APG~\cite{tseng2008accelerated} or FISTA~\cite{beck2009fista} enjoys linear convergence~\cite{FercoqQuDeter17} and can have improved complexity bound than proximal gradient method with appropriate restart periods~\cite{o2012adaptive}. 
Our goal is to design
restarted APPROX with similar properties based on the results in the previous section.
Having defined a set of integers $\{K_0,K_1,\dots\}$, at which frequencies one wishes to restart
the method, we can write the restarted APPROX 
as in Algorithm~\ref{algo:restartedFGMv1}.

\begin{algorithm}
\begin{algorithmic}[h!]
\STATE Choose $x_0 \in \dom \psi$ and set $\bar x_0 = x_0$.
\STATE Choose restart periods $\{K_0,\dots,K_r,\dots\}\subseteq \mathbb N$.
\FOR{ $r \geq 0$}
\STATE $\bar x_{r+1}$ = APPROX($f, \psi, \bar x_r, K_r$)
\ENDFOR 
\end{algorithmic}
\caption{APPROX with restart}
\label{algo:restartedFGMv1}
\end{algorithm}

In order to take profit of the local error bound, 
the proofs for restarted FISTA and APG used the non-blowout property guaranteeing that
the function value and distance to the optimal points
never go above their initial value~\cite{FercoqQuDeter17}.
For APPROX 
the complexity result holds in expectation only.
This has two consequences:
\begin{itemize}
\item Even if the initial point is in a set where the
local error bound holds, we are not guaranteed that the next iterates
will remain to the set. To overcome this issue, at the time
of restart, we will check whether the function value has increased
and if this is unfortunately the case, we will instead restart at 
the initial point.
\item Designing an adaptive restart scheme is much more complex than 
in the deterministic case studied  in~\cite{FercoqQuDeter17,nesterov2013gradient}. In the deterministic case, one can compare the actual progress of the method with a theoretical bound 
and, if the actual progress violates the theoretical bound, one has a certificate
that the estimate of the local error bound was not correct. In the random case,
this does not hold any more: one can only know that either the local error bound was not correct or we have fallen into a small probability event. 
Instead of looking for certificates, we introduce in Section~\ref{sec:variable_restart} a sequence of variable restart periods that allows
us to try several restart periods at an expense that we control.

\end{itemize} 

To partially overcome the above mentioned difficulties, 
it will be convenient to force a decrease in function value
when a restart takes place. This leads to Algorithm~\ref{algo:restartedFGMv2}. Note that we only check function values at the time of each restart: thus this means negligible overhead.

\begin{algorithm}
\begin{algorithmic}[h!]
\STATE Choose $x_0 \in \dom \psi$ and set $\tilde x_0=x_0$.
\STATE Choose restart periods $\{K_0,\dots,K_r,\dots\}\subseteq \mathbb N$.
\FOR{ $r \geq 0$}
\STATE $\bar x_{r+1}$ = APPROX($f, \psi, \tilde x_r, K_r$)
\STATE
$\tilde x_{r+1} \leftarrow \bar x_{r+1} 1_{F(\bar x_{r+1})\leq F(\tilde x_r)}+\tilde x_r 1_{F(\bar x_{r+1})> F(\tilde x_r)}$
\ENDFOR 
\end{algorithmic}
\caption{APPROX with restart and guaranteed function value decrease}
\label{algo:restartedFGMv2}
\end{algorithm}

\section{Linear convergence with constant restart periods}
\label{sec:restart}
In this section we consider constant restart periods, i.e. $K_i=K$ for all $i\in \bN$. We present two types of convergence result for restated APPROX. The first one  asserts that  linear convergence (in expectation) can be obtained for restarted APPROX with long enough restart period $K$, similar to  the classical results about restarted APG or FISTA~\cite{nesterov2013gradient,o2012adaptive}. The second one claims that linear convergence (in expectation) is guaranteed for arbitrary restart period, if strong convexity condition holds.
The basic tool upon which we build our analysis is
a contraction property.

\subsection{Contraction for long enough periods}

The first result is an extension of the ``optimal fixed restart''
of \cite{nesterov2013gradient,o2012adaptive} to APPROX.

\begin{proposition}[Conditional restarting at $x_k$]
Let $(x_k, z_k)$ be the iterates of APPROX applied to~\eqref{eq-prob}. 
Denote:
$$
\tilde x = x_k 1_{F(x_k)\leq F(x_0)}+x_0 1_{F(x_k)> F(x_0)}.
$$
 We have
\begin{align}\label{a:Ftildexk}
\Exp[F(\tilde x) - F^*]\leq
\theta_{k-1}^2 \left(\frac{1 - \theta_0}{\theta_{0}^2} + \frac{1}{\theta_0^2\mu(v,x_0)} \right)(F(x_{0}) - F^*) .
\end{align}
Moreover, given $\alpha < 1$, if 
\begin{align}\label{a:conditionalresxk}
k \geq \frac{2}{\theta_0}\left(\sqrt{\frac{1+\mu(v,x_0)}{\alpha\mu(v,x_0)}}-1\right) + 1 ,
\end{align}
then $\Exp[F(\tilde x) - F^*] \leq \alpha (F(x_{0}) - F^*)$.
\label{prop:restart_x}
\end{proposition}
\begin{proof}By Proposition~\ref{prop:approx_basic}, the following holds for the iterates of APPROX:
\begin{align*}
\Exp[F(x_k) - F^*]& {\leq}
\theta_{k-1}^2 \left(\frac{1 - \theta_0}{\theta_{0}^2}(F(x_{0}) - F^*) + \frac{1}{2\theta_0^2} \dist_v(x_{0}, \mathcal X^*)^2 \right) \\
& \overset{\eqref{a:strconv2}}{\leq} 
\theta_{k-1}^2 \left(\frac{1 - \theta_0}{\theta_{0}^2} + \frac{1}{\mu(v,x_0)\theta_0^2} \right)(F(x_{0}) - F^*). 
\end{align*}
As with probability one,
$$
F(\tilde x) - F^*\leq F(x_k)-F^*,
$$
we obtain~\eqref{a:Ftildexk}.
Condition~\eqref{a:conditionalresxk} is equivalent to:
$$
\frac{4}{(k-1+2/\theta_0)^2}\left(\frac{1 }{\theta_{0}^2} + \frac{1}{\mu(v,x_0)\theta_0^2} \right) \leq \alpha,
$$
and we have the contraction using~\eqref{athetabd}.
\end{proof}
\begin{remark} \normalfont
Notice that the condition~\eqref{a:conditionalresxk} requires to know a lower bound on the error bound constant $\mu(v,x_0)$.
\end{remark}

\begin{corollary}
\label{cor:fixed_restart}
Denote $K(\alpha) = \left\lceil \frac{2}{\theta_0}\left(\sqrt{\frac{1+\mu(v,x_0)}{\alpha\mu(v,x_0)}}-1\right) + 1 \right\rceil$. 
If the restart periods $\{K_0, \cdots, K_r,\dots\}$ are all equal to $K(\alpha)$,
then the iterates of Algorithm~\ref{algo:restartedFGMv2} satisfy
\[
\Exp[F(\tilde x_r) - F^*] \leq \alpha^r (F(x_0) - F^*).
\]
\end{corollary}
\begin{proof}
By the definition of $\tilde x_r$, we know that for all $r$, 
$F(\tilde x_r) \leq F(x_0)$. Hence, for all $r$, $\mu(v, \tilde x_r) \geq \mu(v, x_0)$.
We can thus apply Proposition~\ref{prop:restart_x} recursively and obtain the linear convergence.
\end{proof}
The iterate $\tilde x_{r}$ is obtained after running $r$ times APPROX (Algorithm~\ref{alg:approx}) of  $K(\alpha)$ iterations.
Said otherwise we have a linear rate equal to
\begin{align}\label{a:rate}
\alpha^{1/K(\alpha)} \approx \alpha^{\sqrt{\alpha} \frac{\theta_0}{2} \sqrt{{\mu(v,x_0)}/(1+\mu(v,x_0))}},
\end{align}
which suggests choosing $\alpha = \exp(-2)$, or equivalently a fixed restart period
\begin{align}\label{a:Kstar}
K^*= \left\lceil \frac{2e}{\theta_0}\left(\sqrt{\frac{1+\mu(v,x_0)}{\mu(v,x_0)}}-1\right) + 1 \right\rceil.
\end{align}
Then, to obtain $\Exp[F(\tilde x_r)-F^*]\leq \epsilon$, the total number of iterations is bounded by
\begin{align}\label{a:sqrtmu}
N^*:=\ln\left(\frac{F(x_0)-F^*}{\epsilon}\right) K^*.
\end{align}

\begin{remark}
We compare the two extreme cases of the complexity bound~\eqref{a:sqrtmu} when $\tau=1$ and $\tau=n$. 
In view of~\eqref{a:muv} and~\eqref{eq:ESO-2}, we have,
$$
O\left(\sqrt{\frac{L}{\mu(x_0)}}\ln(1/\epsilon)\right) \mathrm{~for~} \tau=n \mathrm{~V.S.~} 
O\left(n\sqrt{\frac{\max_i L_i}{\mu(x_0)}}\ln(1/\epsilon)\right) \mathrm{~for~} \tau=1,
$$
where $L$ is the Lipschitz constant of $\nabla f$ and $L_i$ is the Lipschitz constant of $\nabla_i f$ with respect to the $i$th coordinate.
Note that $\tau$ is the number of coordinates to update at each iteration.
Hence for serial computation, choosing $\tau=1$ can be more advantegeous 
than restarted APG or FISTA since
we always have $\max_i L_i\leq L$.
\end{remark}

\begin{remark}\label{remun}
It is unknown whether Algorithm~\ref{algo:restartedFGMv1} admits the same convergence bound as Algorithm~\ref{algo:restartedFGMv2}. Indeed,
without forcing decrease of the objective value, we can not guarantee $\Exp[F(\bar x_{r+1}) - F^*] \leq \alpha \Exp[F(\bar x_{r}) - F^*]$ as $\bar x_r$ is not necessarily in the sublevel set $[F\leq F(x_0)]$. 
Clearly, when either a global error bound condition holds or $\hat S=[n]$ with probability one, there is no need to ensure the decrease of $F$. The latter two cases were respectively considered in our two previous papers~\cite{fercoqqu2016restarting} and~\cite{FercoqQuDeter17}.
\end{remark}

\subsection{Contraction for any period under strong convexity condition}

In this subsection, we assume that the function $F$ is $\mu$-strongly convex such that $\cX^*$ contains a unique element $x^*$ and
\begin{align}\label{a:geb}
F(x)\geq F^* +\frac{\mu}{2}\|x-x^*\|_v^2, \enspace \forall x\in [F<+\infty].
\end{align}
In~\cite{fercoqqu2016restarting}, we showed that under the condition~\eqref{a:geb}, linear convergence is guaranteed for any restart period by restarting at a particular convex combination of all past iterates. Here we show that the same conclusion holds by simply restarting at the last iterate.

\begin{theorem}
\label{th:strconv}
Assume that $F$ is $\mu$-strongly convex. Denote 
\[
\Delta(x) = \frac{1-\theta_0}{\theta_0^2}(F(x) - F^*) + \frac{1}{2\theta_0^2}\dist_v(x, \cX^*)^2.
\]
Then the iterates of APPROX satisfy
\begin{align}\label{a:delta}
\Exp[\Delta(x_k)] \leq \frac{1 + (1-\theta_0)\mu}{1 + \frac{\theta_0^2}{2\theta_{k-1}^2}\mu} \Delta(x_0),\enspace \forall k\geq 1.
\end{align}
\end{theorem}
\begin{proof}
The proof is organised in 4 steps. Firstly, we derive a one-iteration inequality,
secondly, we identify conditions under which this inequality is a supermartingale inequality,
thirdly, we give a solution to the set of inequalities and finally, we bound the rate we obtain
by a simpler formula.

Because of its length, we postpone the proof to the appendix.
\end{proof}
\begin{remark}\label{remde}
The deterministic special case of~\eqref{a:delta} when $\hat S=[n]$ with probability one, was proved in~\cite{FercoqQuDeter17}, under the local error bound condition. The more general case with random $\hat S$ turns out to be more complexe as the one iteration inequality involves both $\|x_k-x^*\|_v$ and $\|z_k-x^*\|_v$ and uniqueness of $\cX^*$ is required.
\end{remark}

When strong convexity condition holds, we do not force decrease of the function value after a restart. Moreover, unlike Corollary~\ref{cor:fixed_restart}, we could show a linear rate of convergence for any restart period $K\geq n/\tau$. 
\begin{corollary}
Under condition~\eqref{a:geb},  if the restart periods $\{K_0, \cdots, K_r,\dots\}$ are all equal to $K\geq n/\tau$, then the iterates of Algorithm~\ref{algo:restartedFGMv1} satisfy
\[
\Exp(\Delta(\bar x_r))  \leq \left(\frac{1 + (1-\theta_0)\mu}{1 + \frac{\theta_0^2}{2\theta_{K-1}^2}\mu}\right)^{r} \Delta(x_0).
\]
\end{corollary}
\begin{proof}
This is a direct application of Theorem~\ref{th:strconv}. The bound~\eqref{a:delta} is a strict contraction when $2\theta^2_{k-1}< \theta^2_0/(1-\theta_0)$, which holds when $k\geq 1/\theta_0=n/\tau$ by~\eqref{athetabd}.
\end{proof}
We here obtain a linear rate:
\begin{align}\label{a:deltarate}
\left(\frac{1 + (1-\theta_0)\mu}{1 + \frac{\theta_0^2}{2\theta_{K-1}^2}\mu}\right)^{\frac 1K},
\end{align}
which is slightly worse than that suggested by~\eqref{a:rate} for large $K$
 but implies the same order of complexity bound as~\eqref{a:sqrtmu} 
if $K=\Theta(K^*)$.

In Figure~\ref{fig:rates_strong_conv}, we show in a numerical application that the convergence rate~\eqref{a:deltarate} of 
restarted APPROX is smaller 
than the convergence rate of coordinate descent for a wide range of
restart periods $K$.

\begin{figure}
\centering

\includegraphics[width=0.7\linewidth]{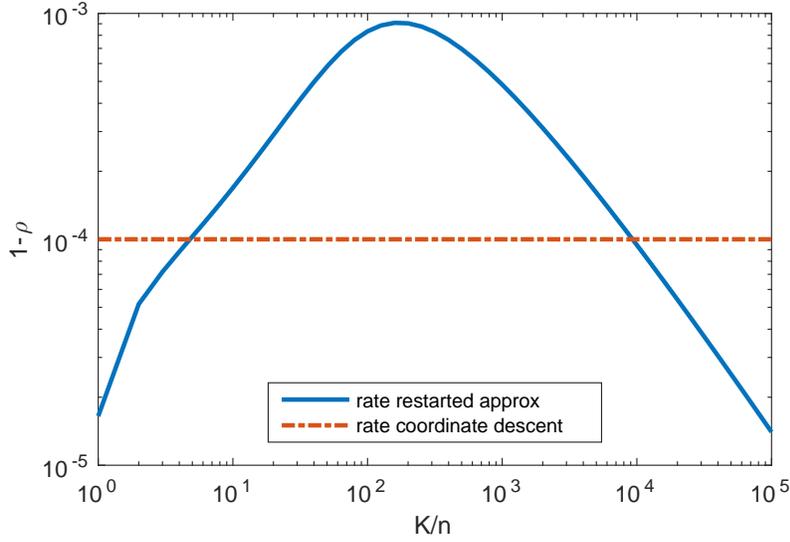}

\caption{Comparison of the worst case rate of convergence $\rho$ of restarted APPROX and coordinate descent for a $\mu$-strongly convex objective with $\mu = 10^{-3}$, $n=10$, $\tau = 1$ and various choices of the restart period~$K$. The larger $1-\rho$ is, the faster we expect the algorithm to be. We can see that if $K \in [5 n, 9.10^3 n]$, then 
restarted APPROX has a better convergence rate than coordinate descent.}
\label{fig:rates_strong_conv}
\end{figure}

\section{Variable restart}
\label{sec:variable_restart}

In this section, we are going to show that by choosing properly the
sequence of restart times, we can ensure a nearly linear rate, even if
we  know in advance nothing about the local quadratic error bound.
For this, we consider a sequence of restart periods $\{K_0, K_1,\cdots\}$ satisfying the following assumption:
\begin{ass}\label{ass:K}
~

\begin{enumerate}
\item
$K_0\in \bN\backslash \{0\}$;
\item
$K_{2^j-1} = 2^j K_0$ for all $j \in \mathbb{N}$;
\item
$\left |\{0\leq r < 2^J - 1 \;|\; K_r= 2^j K_0 \}\right | = 2^{J-1-j}$ for all $j \in \{0,1, \ldots, J-1\}$ and $ J\in \mathbb{N}$.
\end{enumerate}
\end{ass}
\medskip
For instance, we may take:
\begin{align*}
&K_0 = K_0 \quad K_1 = 2^1 K_0 \quad K_2 = K_0 \quad K_3 = 2^2 K_0 \\ &
K_4 = K_0 \quad K_5 = 2^1 K_0 \quad K_6 = K_0 \quad K_7 = 2^3 K_0 \\
&K_8 = K_0 \quad K_9 = 2^1 K_0 \quad K_{10} = K_0 \quad K_{11} = 2^2 K_0 
\\&K_{12} = K_0 \quad  K_{13} = 2^1 K_0 \quad K_{14} = K_0 \quad K_{15} = 2^4 K_0 
\end{align*}

\begin{theorem}
\label{thm:variable}
Consider Algorithm~\ref{algo:restartedFGMv2} with restart periods satisfying Assumption~\ref{ass:K}. 
Then 
\begin{align}\label{a:Fxepsilon}
\Exp[F(\tilde x_{2^J-1}) - F^*] \leq \epsilon,\end{align}
where $J=\ceil{{\max\left(\log_2\left({K^*}/{K_0}\right), 0\right)}} + \ceil{{\log_2\left({\ln\left({\delta_0}/{\epsilon}\right)/2}\right)}}$, $\delta_0=F(x_0)-F^*$ and $K^*$ is defined as in~\eqref{a:Kstar}. To obtain~\eqref{a:Fxepsilon}, the total number of APPROX iterations $K_0+\dots+K_{2^J-1}$ is bounded by
\begin{align}\label{a:boundK}
\big(\ceil{\max\left(\log_2\left({K^*}/{K_0}\right), 0\right)} + \ceil{\log_2\left({\ln\left({\delta_0}/{\epsilon}\right)}\right)}+1\big) {\ln\left({\delta_0}/{\epsilon}\right)}\max(K^*, K_0).
\end{align}
\end{theorem}
\begin{proof}
Define
$$
i =\ceil{\max\left(\log_2\left({K^*}/{K_0}\right), 0\right)}.
$$
Denote $\delta_r = \Exp[F(\tilde x_{r}) - F^*]$ and
\begin{align*}
&c_i(J):=\left |\left\{ l < 2^J-1 |  K_l \geq 2^i K_0\right\}\right | + 1 = 1 + \sum_{k=i}^{J-1} 2^{J-1-k} = 2^{J-i}.
\end{align*}
$c_i(J)$ is the number of restarts such that $K_l \geq K^*$, i.e.\ the number of 
restarts for which Proposition~\ref{prop:restart_x} applies (the ``$+1$'' comes from the restart number $2^J-1$).
Then it follows from Proposition~\ref{prop:restart_x}  that 
$$
\delta_{2^J-1}\leq e^{-2c_i(J)} \delta_0= e^{-(2^{J+1-i})} \delta_0.
$$
Since 
$$
J = \ceil{\max\left(\log_2\left({K^*}/{K_0}\right), 0\right)} + \ceil{\log_2\left({\ln\left({\delta_0}/{\epsilon}\right)}\right)} - 1 \geq  i-1+ {\log_2\left({\ln\left({\delta_0}/{\epsilon}\right)}\right)},$$
we deduce that $\delta_{2^{J}-1}\leq \epsilon$. This proves the first assertion.

By Assumption~\ref{ass:K},
\begin{align*}
K_0+\dots+K_{2^J-1}&=\sum_{j=0}^{J-1} \left |\{0\leq r < 2^J - 1 \;|\; K_r= 2^j K_0 \}\right | 2^jK_0+K_{2^J-1}\\
&=\sum_{i=0}^{J-1}2^{J-1-j}2^j K_0+2^J K_0=(J+2)2^{J-1}K_0.
\end{align*}
Since
$$
2^{J-1}\leq 2^{{\max\left(\log_2\left({K^*}/{K_0}\right), 0\right)} + {\log_2\left({\ln\left({\delta_0}/{\epsilon}\right)}\right)}}=\max(K^*/K_0,1)\ln(\delta_0/\epsilon),
$$
we obtain the bound defined as in~\eqref{a:boundK}.
\end{proof}

\begin{remark}
If we have an estimate of $\log(\delta_0/\epsilon)$, then we may modify the sequence of restarts in order to stop considering restarts with $K_0$ iterations after $\log(\delta_0/\epsilon)$ restarts. Indeed, if $K_0$ were greater than $K^*$ then the algorithm should already have terminated. Similarly, one can stop considering restarts with $2 K_0$ iterations after $2 \log(\delta_0/\epsilon)$ restarts, $2^j K_0$ after $(j+1) \log_2(\delta_0/\epsilon)$ restarts.
\end{remark}
%
%
%
%
%

%

We summarize our theoretical findings in Table~\ref{tab:comprates} in four different regimes: when we know a lower bound $\underline{\mu}$, 
an upper bound $\bar \mu$, exactly the value of or nothing about  $\mu(v,x_0)$. 
We recall from~\eqref{a:estimationofmu} 
and~\eqref{a:thetaformu} that a non-trivial lower bound $\underline{\mu}$ can be much harder to be obtained than an upper bound   $\bar \mu$. In addition,  the complexity bound based on an upper bound $\bar \mu$
differs from the optimal scheme by a logarithm term.
\begin{table}
\centering
\resizebox{\textwidth}{!}{
\begin{tabular}{|c|c|c|c|}
\hline A priori knowledge& Theorem & Assumption  & Complexity bound \\
\hline $\mu(v,x_0)$ &  Corollary~\ref{cor:fixed_restart}& Assumption~\ref{ass2} &
$N^* $ \\
\hline $ \rule{0pt}{4ex}   \mu(v,x_0)\geq \underline{\mu} $ & Proposition~\ref{prop:restart_x}& Assumption~\ref{ass2} &
$O\left(N^*\sqrt{\frac{\mu(v,x_0)}{\underline{\mu}}}\right)$
\\
\hline \rule{0pt}{4ex}   $\mu_0>0$ & Theorem~\ref{th:strconv} & ~\eqref{a:geb} &
$O\left(N^*\left(\sqrt{\frac{\mu}{\mu_0}}+\sqrt{\frac{{\mu_0}}{\mu}}\right)\right)$
\\
\hline \rule{0pt}{4ex}   $ \mu(v,x_0)\leq \bar \mu$ &  Theorem~\ref{thm:variable} & Assumption~\ref{ass2}+~\ref{ass:K} &
$O\left(N^*\log_2\left(\ln\left(\frac{F(x_0)-F^*}{\epsilon}\right)\sqrt{\frac{\bar \mu}{{\mu(v,x_0)}}}\right)\right)$
\\ \hline\end{tabular}}
\caption{Complexity bound under different assumptions and a priori knowledge of the error bound constant. Here 
$N^*=O\left(\frac{\ln(1/\epsilon)}{\theta_0\sqrt{\mu(v,x_0)}}\right)$  is defined in~\eqref{a:sqrtmu}.}
\label{tab:comprates}
\end{table}
For comparison purpose, we recall  the log-scale grid search schedule proposed in~\cite{RouletAspremontNIPS}
for restart of APG or FISTA.
Fixing an integer $N>0$, the restart periods proposed in~\cite{RouletAspremontNIPS} can be described as follows:
\begin{align*}
&K_0=\dots=K_{\ceil{N/2}}=2\\
&K_{\ceil{N/2}+1}=\dots=K_{\ceil{N/2}+\ceil{N/2^2}}=2^2
\\& \small\vdots\\
&K_{\ceil{N/2}+\dots\ceil{N/2^{i-1}}+1}=\dots=K_{\ceil{N/2}+\dots+\ceil{N/2^i}}=2^i
\end{align*} 
To ensure a nearly linear time convergence with logarithm difference, 
the inner number of iterations $N$ is required to be larger than $2e\sqrt{L/\mu(x_0)}$. Therefore a lower bound $\underline\mu \leq \mu(x_0)$ is needed.

\section{Extension to other randomized accelerated methods}
\label{sec:acc-svrg}

Proposition \ref{prop:restart_x} and Theorem~\ref{thm:variable} only rely on
the fact that APPROX guarantees 
\begin{equation}
\label{eq:generalization}
\Exp[F(x_k)-F^*] \leq \frac{1}{(k+a)^2}\Big(C_F(F(x_0) - F^*) + \frac{C_d}{ 2} \dist(x_0, \mathcal X^*)^2 \Big)
\end{equation}

A couple of other algorithms have such convergence guarantees. Instead of assuming that $\psi$ is separable, they assume that $f$ is a sum of functions $f(x) = \frac{1}{2n_f}\sum_{j=1}^{n_f} f_j(x)$:
\begin{itemize}
	\item Katyusha$^{\text{ns}}$~\cite{allen2017katyusha} (an accelerated stochastic variance reduced gradient method)
	satisfies for its output vector $\tilde x^S$
	\begin{align*}
	\Exp[F(\tilde x^S) - F^*] \leq \frac{8}{(S+3)^2}\Big(2(F(x_0) - F^*) + \frac{3L}{2m} \dist(x_0, \mathcal X^*)^2 \Big)
	\end{align*}
	where $m = 2n_f$ is the number of SVRG steps 
	between each momentum step, $n_f$ is the number
	of summands in the definition of $f$ and the other symbols follow the notation of this paper.
	
	\item DASVRDA$^{\text{ns}}$~\cite{murata2017doubly} (another accelerated stochastic variance reduced gradient method) has a similar guarantee
	\begin{multline*}
	\Exp[F(\tilde x^S) - F^*] \leq \frac{4}{(S+2)^2}\Big((F(x_0) - F^*) \\
	+ \frac{2(1+\gamma (m+1))L}{(1-\gamma^{-1})m(m+1)} \dist(x_0, \mathcal X^*)^2 \Big)
	\end{multline*}
	where $\gamma \geq 3$ is a parameter of the method.
	
\end{itemize}

From~\eqref{eq:generalization} we obtain directly:
\[
	\Exp[F(x_k)-F^*] \leq \frac{1}{(k+a)^2}\Big(C_F + \frac{C_d}{\mu( x_0)}\Big)\Big(F(x_0) - F^* \Big).
	\]
which is a strict contraction if 
\begin{equation}
\label{eq:kstar_general}
k\geq K^* := \left\lceil \frac{1}{e}\sqrt{C_F + \frac{C_d}{\mu(x_0)}}-a\right \rceil \;.
\end{equation} 
Then 
the same conlusion as Corollary~\ref{cor:fixed_restart} and Theorem~\ref{thm:variable} can be obtained.

\begin{theorem}
	Let $\cA$ be an algorithm satisfying \eqref{eq:generalization}.
Consider restarted $\cA$  in the fashion of Algorithm~\ref{algo:restartedFGMv2}.  Then, to obtain 
$$
\Exp[F(\tilde x_r)-F^*]\leq \epsilon,
$$
the number of inner iterations is bounded by
$$
\ln\left(\frac{F(x_0)-F^*}{\epsilon}\right) K^*,
$$ 
	if the restart periods  are all equal to $K^*$ defined in~\eqref{eq:kstar_general}, and by
$$
\big(\ceil{\max\left(\log_2\left({K^*}/{K_0}\right), 0\right)} + \ceil{\log_2\left({\ln\left({\delta_0}/{\epsilon}\right)}\right)}+1\big) {\ln\left({\delta_0}/{\epsilon}\right)}\max(K^*, K_0)
$$
if the restart periods  satisfy Assumption~\ref{ass:K}.
\end{theorem}
\begin{proof}
	The proof arguments are the same as in the coordinate descent case.
\end{proof}

\begin{remark}
The proof of Theorem~\ref{th:strconv} requires to go deeper into the 
specificities of accelerated coordinate descent. Hence, extending restart at any fixed frequency to
any randomized accelerated method is not trivial, and beyond the scope of this paper.
\end{remark}

\section{Numerical experiments}
\label{sec:expe}

\subsection{Logistic regression}


We solve the following logistic regression problem:
\begin{align}\label{alr}
\min_{x \in \mathbb{R}^N} \frac{\lambda_1}{\|A^\top b\|_{\infty}} \sum_{j=1}^m \log(1 + \exp(b_j a_j^\top x)) + \norm{x}_1 
\end{align}
We consider 
$$
f(x)= \frac{\lambda_1}{ \|A^\top b\|_{\infty}} \sum_{j=1}^m \log(1 + \exp(b_j a_j^\top x)) ,
$$
and 
$$
\psi(x)=\norm{x}_1. 
$$
In particular, for serial sampling ($\tau=1$),~\eqref{eq:ESO} is satisfied for
\begin{align}\label{av}
v_i=\frac{\lambda_1}{ 4\|A^\top b\|_{\infty}}\sum_{i=1}^m (b_jA_{ij})^2,\enspace i=1,\dots,n.
\end{align}

Even if the logistic loss
is not strongly convex, we expect that the local curvature around the optimum is nonzero
and so, restarting APPROX may be useful.

We run our numerical comparison on the dataset RCV1~\cite{CC01a} with regularization parameter $\lambda_1 = 10^4$. We compare the following algorithms.
\begin{itemize}
\item Randomized coordinate descent~\cite{RT:UCDC}.
\item APPROX~\cite{FR:2013approx}.
\item APCG~\cite{lin2014accelerated}: we run APCG using three different settings for the parameter $\mu$ in the algorithm: $10^{-3}$, $10^{-5}$ and $10^{-7}$. We stop the program when the duality gap is lower than $10^{-8}$ or the running time is larger than 6,000s. The results are reported in Figure~\ref{fig:d0}. 
\item Prox-Newton: we modified the implementation of~\cite{johnson2015blitz} in order to deactivate feature selection, which is not the topic of the present paper, and we increased the maximum number of inner iterations to be able to obtain high accuracy solutions. Each prox-Newton step is solved approximately using coordinate descent.
\item APPROX-restart (Algorithm~\ref{algo:restartedFGMv2}) with fixed frequency set as if we knew the error bound constant. As for APCG, we tried values equal to $10^{-3}$, $10^{-5}$ and $10^{-7}$. 
\item APPROX-restart (Algorithm~\ref{algo:restartedFGMv2})
with the restart sequence given in Section~\ref{sec:variable_restart} and $K_0 = \ceil{20 e / \theta_0}$. Most problems encountered in practice have a conditioning larger than 100,
hence this choice leads to $K_0 \leq K^*$ in most cases.
\end{itemize}


 Note that the convergence of APCG is only proved for $\mu$ smaller than the strong convexity coefficient in~\cite{lin2014accelerated}.  In our experiments, we observed numerical issues when running APCG for several cases when taking larger $\mu$ (we were
 not able to compute the $i$th partial derivative at $y_k = \rho_k w_k + z_k$ because $\rho_k$ had reached the 
 double precision float limit).  Such a case can be identified in the plot when the line corresponding to APCG 
 stops abruptly before the time limit (6000s) with a precision worse than $10^{-8}$.
 
Fixed restart can give the best performance among all tested algorithms
but the setting of the restart period has a large influence on the performance. Note that the objective function is not strongly convex, so 
the linear convergence is guaranteed only if the restart period is large enough. 

On the other hand, variable restart APPROX is nearly as fast as the best among the fixed restart, although $K_0$ was set with a clearly under optimal value. Moreover, it has a clear theoretical guarantee for any initial restarting period $K_0$.

\begin{figure}[htbp]
\centering
\includegraphics[width=\linewidth, trim=8em 8ex 8em 6ex, clip]{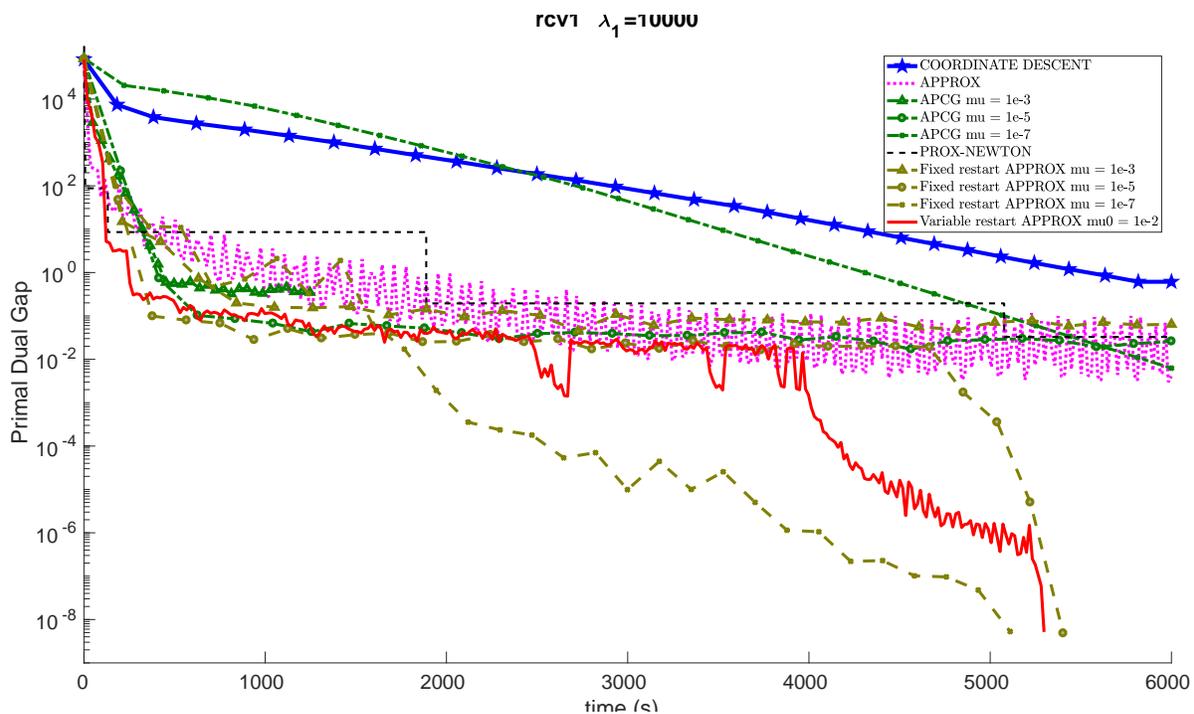}
\caption{Comparison of (accelerated) coordinate descent algorithms for the logistic regression problem on the dataset RCV1: coordinate descent, APPROX, APCG with $\mu \in \{10^{-3}, 10^{-5}, 10^{-7}\}$, Prox-Newton, fixed frequency restarted APPROX with $\mu \in \{10^{-3}, 10^{-5}, 10^{-7}\}$ and APPROX with variable restart initiated with $K_0 = \ceil{20 e / \theta_0}$. APCG failed for $\mu = 10^{-3}$.
}
\label{fig:d0}
\end{figure}

\subsection{Lasso path}

We then present experiments on the $L^1$-regularised least squares problem (Lasso)
\begin{align}\label{a:LASSO}
\min_{x \in \mathbb{R}^N} \frac{1}{2} \norm{A x - b}^2_2 + \lambda \norm{x}_1.
\end{align}
We consider solving a set of such problems for $\lambda \in \{\lambda_0, \lambda_1, \ldots, \lambda_T\}$, where $T = 10$, $\lambda_0 = \norm{A^\top b}_{\infty}$, $\lambda_t = \lambda_0 \alpha^t$ and
$\alpha^T = 10^{-3}$. This procedure is called pathwise optimization~\cite{friedman2007pathwise} and is often considered for statistical problems with hyper-parameters.

We selected 6 data sets from the LibSVM dataset repository~\cite{CC01a}. We
centered and normalized the columns. 

We did not run APCG on this problem because, as shown on the logistic regression experiment, the strong convexity estimate has a dramatic consequence on the behaviour of the algorithm and setting it to a too high value may lead to major numerical issues.

We chose to restart APPROX with variable restart set as follows. 
In the beginning, we start with $10n$ iterations of non-accelerated coordinate descent. After that, we run variable restart APPROX with $K_0 = 10n$ and we double $K_0$ after each $\log_2(1/\epsilon)$ iterations. 
When the duality gap at $\lambda_t$ reaches $\epsilon$, we 
switch to $\lambda_{t+1}$. We perform a warm start on the 
optimization variable but we also set $K_0$ to the 
last value it had when solving the problem at $\lambda_t$.

We then compare coordinate descent~\cite{RT:UCDC}, APPROX~\cite{FR:2013approx} and
restarted APPROX (Algorithm~\ref{algo:restartedFGMv2}) on the 6 pathwise optimization lasso problems.

APPROX is sometimes getting into trouble when high accuracy is requested. This does not happen with variable restart APPROX. In all the experiments, variable restart APPROX is at most twice as slow as coordinate descent. On the other hand, variable restart APPROX is much faster on problems requiring more computational resources.

\subsection{Lasso with SVRG}

Our last experiment illustrate the restart of accelerated stochastic variance reduced 
gradient on a Lasso problem~\eqref{a:LASSO}. We took the \texttt{abalone} dataset and $\lambda = \lVert A^\top b \rVert_{\infty} / 1000$ and solved the problem using 
SVRG~\cite{allen2016improved}, Katyusha$^{\text{ns}}$~\cite{allen2017katyusha} and restarted variants of Katyusha$^{\text{ns}}$ described in Section~\ref{sec:acc-svrg}.
As shown on Figure~\ref{fig:acc-svrg}, restarted accelerated SVRG is able to 
solve the problem faster than previously proposed methods.

\begin{figure}
\centering
\includegraphics[width=0.6\linewidth]{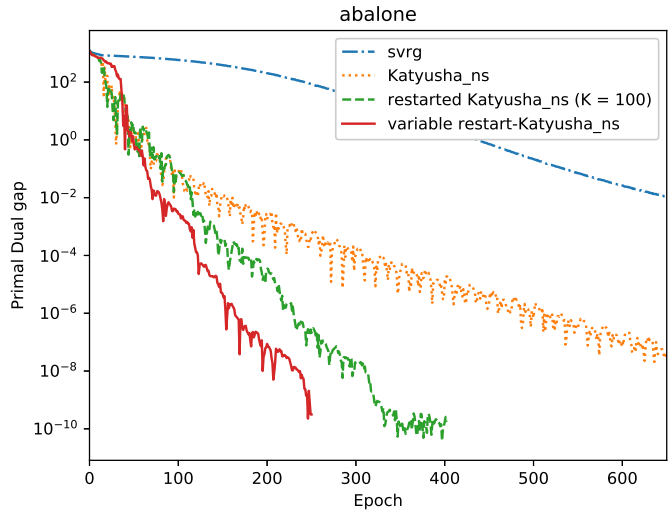}
\caption{Comparison of vanilia, accelerated and restarted versions of SVRG. An epoch corresponds to $m$ SVRG steps.}
\label{fig:acc-svrg}
\end{figure}

\begin{table}
\small

\centering

\begin{tabular}{|l|rrrr|}
\hline
Data set    &      Accuracy  &     Coordinate  &  APPROX  &  Variable restart   \\
            &                & descent  &        & APPROX \\
\hline
abalone & $10^{-2}$&  0.066 &  0.045    &   0.047  \\
$n$ = 8 &$10^{-6}$& 0.211 &   \bf 0.126 &    0.133  \\
$m$ = 4,177 &$10^{-10}$&  0.404  &  0.271  &  \bf   0.257\\
\hline 
triazines & $10^{-2}$&  0.194 & \bf 0.038  &  0.060   \\
$n$ = 60 &$10^{-6}$& \em $>$2.442  &   0.989 & \bf 0.379   \\
$m$ = 186 &$10^{-10}$&  \em $>$3.788 & \em $>$9.850   &   \bf 1.038 \\
\hline
leukemia &   $10^{-2}$ & \bf 0.113 & 0.174  &   0.189  \\ 
$n$ = 7129 &  $10^{-6}$&  10.367  &  11.735 & \bf 5.132  \\
$m$ = 72  &   $10^{-10}$&  97.532 &  \em $>$435.002   & \bf  23.223 \\
\hline
blogfeedback&   $10^{-2}$  &  6.732 & \bf 2.561  & 3.868   \\
$n$ = 280 &  $10^{-6}$&  452.177 &  63.832 &  \bf 54.962 \\
$m$ = 52,397 &   $10^{-10}$& \em $>$1617.721 & \em $>$1845.049  & \bf  208.319 \\
\hline
rcv1 & $10^{-2}$&  \bf 2.898 & 4.381 &  4.323   \\
$n$ = 47236&$10^{-6}$& 227.590 &  108.778& \bf 69.616   \\
$m$ = 20242&$10^{-10}$&    2986.070& 943.178  & \bf  180.206  \\
\hline
news20  & $10^{-2}$& {\bf 33}\phantom{.000} & 64\phantom{.000} & 67\phantom{.000}   \\
$n$ = 1,355,191&$10^{-6}$& 8161\phantom{.000} & 2594\phantom{.000} & {\bf 1869}\phantom{.000}   \\
$m$ = 19,996&$10^{-10}$& \em $>$36000\phantom{.000}& \em $>$82000\phantom{.000} & {\bf 7173}\phantom{.000} \\
\hline
\end{tabular}

\caption{Time is seconds to compute the Lasso path over a grid. The $>$ sign means that the algorithm had not reach the target accuracy on the duality gap after 40,000 $n$ coordinate updates, for at least one $\lambda_t$ in the grid. We put the number in boldface when an algorithm is at most 50\% faster than another.}
\end{table}

\appendix

\section{Proof of Lemma~\ref{l:thetak} and Theorem~\ref{th:strconv}}

\begin{proof}[Proof of Lemma~\ref{l:thetak}]
The equation~
\eqref{arectheta} holds because $\theta_{k+1}$ is the unique positive square root
to the polynomial $P(X) = X^2 + \theta_k^2 X - \theta_k^2$. \eqref{atheradecr} is a direct
consequence of~\eqref{arectheta}.

Let us prove \eqref{athetabd} by induction. 
It is clear that $\theta_0 \leq \frac{2}{0+2/\theta_0}$. 
Assume that $\theta_k \leq \frac{2}{k+2/\theta_0}$.
We know that $P(\theta_{k+1}) = 0$ and that $P$ is an increasing function on $[0 , +\infty]$. So we just need to show that $P\big(\frac{2}{k+1+2/\theta_0}\big)\geq 0$.
\begin{align*}
P\Big(\frac{2}{k+1+2/\theta_0}\Big) = \frac{4}{(k+1+2/\theta_0)^2} 
+ \frac{2}{k+1+2/\theta_0}\theta_k^2 - \theta_k^2
\end{align*}
As $\theta_k \leq \frac{2}{k+2/\theta_0}$ and $\frac{2}{k+1+2/\theta_0}-1 \leq 0$,
\begin{align*}
P\Big(\frac{2}{k+1+2/\theta_0}\Big) &\geq \frac{4}{(k+1+2/\theta_0)^2} 
+ \Big(\frac{2}{k+1+2/\theta_0} - 1\Big) \frac{4}{(k+2/\theta_0)^2} \\
 &= \frac{4}{(k+1+2/\theta_0)^2 (k+2/\theta_0)^2} \geq 0.
\end{align*}

For the other inequality,  $\frac{(2-\theta_0)}{0+(2-\theta_0)/\theta_0} \leq \theta_0$.
We now assume that $\theta_k \geq \frac{(2-\theta_0)}{k+(2-\theta_0)/\theta_0}$ but
that $\theta_{k+1} < \frac{(2-\theta_0)}{k+1+(2-\theta_0)/\theta_0}$. 
Remark that $(x \mapsto (1-x)/x^2)$ is strictly decreasing for $x \in (0,2)$. Then, using \eqref{arectheta}, 
we have
\begin{align*} 
\frac{(k+1+(2-\theta_0)/\theta_0)^2}{(2-\theta_0)^2} -\frac{k+1+(2-\theta_0)/\theta_0}{2-\theta_0}< \frac{1-\theta_{k+1}}{\theta_{k+1}^2}  \overset{\eqref{arectheta}}{=} \frac{1}{\theta_k^2} \leq \frac{(k+(2-\theta_0)/\theta_0)^2}{(2-\theta_0)^2}.
\end{align*}
This is equivalent to
\begin{equation*}
(2 - (2- \theta_0)) (k+(2-\theta_0)/\theta_0) + 1 < (2-\theta_0)
\end{equation*}
which obviously does not hold for any $k \geq 0$.
So $\theta_{k+1} \geq \frac{(2-\theta_0)}{k+1+(2-\theta_0)/\theta_0}$.
\end{proof}

\begin{proof}[Proof of Theorem \ref{th:strconv}]
The proof is organised in 4 steps. Firstly, we derive a one-iteration inequality,
secondly, we identify conditions under which this inequality is a supermartingale inequality,
thirdly, we give a solution to the set of inequalities and finally, we bound the rate we obtain
by a simpler formula.

{\em Step 1: Derive a one-iteration inequality.}

Let us denote $\hat F_k = f(x_k) + \sum_{l=0}^k \gamma_k^l \psi(z_l) \geq F(x_k)$ where $\gamma_k^l$ are the same constants as defined in~\cite{FR:2013approx}. By~\cite{FR:2013approx}, for any $x_* \in \cX^*$,
\begin{align}\label{a:Fk}
\Exp_k[\hat F_{k+1}-F^*] \leq (1-\theta_k)(\hat F_k-F^*) + \frac{\theta_k^2}{2 \theta_0^2}\big(\norm{z_k - x_*}_v^2 - \Exp_k[\norm{z_{k+1}-x_*}_v^2]\big)
\end{align}

Using 
\[
x_{k+1} = (1-\theta_k)x_k + \frac{\theta_k}{\theta_0}z_{k+1} - \frac{\theta_k}{\theta_0}(1-\theta_0) z_k
\]
and the equality 
\[
\norm{(1-\lambda)a + \lambda b}^2 = (1-\lambda)\norm{a}^2 + \lambda \norm{b}^2 - \lambda(1-\lambda) \norm{a-b}^2
\]
which is valid for any vectors $a$ and $b$ and any $\lambda \in \mathbb R$, we get, denoting $x_*$ the unique element of $\cX^*$, 
\begin{align}
\label{eq:decompose_xk1}
\norm{x_{k+1}-x_*}_v^2 = (1-\theta_k)\norm{x_k-x_*}_v^2 + \frac{\theta_k}{\theta_0}\norm{z_{k+1}-x_*}_v^2 - \frac{\theta_k}{\theta_0}(1-\theta_0)\norm{z_k - x_*}_v^2 \notag \\
+ \frac{(1-\theta_k)\theta_k/\theta_0(1-\theta_0)}{1-\theta_k/\theta_0}\norm{x_k-z_k}_v^2 - \frac{\theta_k/\theta_0}{1-\theta_k/\theta_0}\norm{x_{k+1}- z_{k+1}}_v^2
\end{align}

Let us consider nonnegative sequences $(a_k)$, $(b_k)$, $(c_k)$, $(d_k)$ and $\sigma^K_k$ and study the quantity 
\begin{multline*}
C_{k+1} =  a_{k+1}(\hat F_{k+1} - F^*) + \frac{b_{k+1}}{2} \norm{x_{k+1} - x_*}_v^2 + \frac{c_{k+1}}{2} \norm{z_{k+1} - x_*}_v^2\\ + \frac{d_{k+1}}{2} \norm{x_{k+1}-z_{k+1}}_v^2
\end{multline*}
By the strong convexity condition of $F$, we have:
\[
\frac 12 \|x_{k+1}-x_*\|_v^2 \leq \frac{1}{\mu}(F(x_{k+1}) - F^*).
\]
Hence, we get
for any $\sigma_{k+1}^K \in [0,1]$, (the usefulness of superscript $K$ will become clear later)
\begin{align*}
\Exp_k[& C_{k+1}] \overset{\eqref{eq:decompose_xk1}}{\leq}  \Exp_k\Big[(a_{k+1}+\frac{b_{k+1}\sigma^K_{k+1}}{\mu})(\hat F_{k+1} - F^*) \\
&+ \frac{b_{k+1}}{2}(1-\sigma^K_{k+1})(1-\theta_k)\norm{x_k-x_*}_v^2 \\
& +  \big(c_{k+1} + b_{k+1}(1-\sigma^K_{k+1})\frac{\theta_k}{\theta_0} \big)\frac 12 \norm{z_{k+1}-x_*}_v^2 \\
& -\frac{b_{k+1}}{2}(1-\sigma^K_{k+1}) \frac{\theta_k}{\theta_0}(1-\theta_0)\norm{z_k - x_*}_v^2 \\
&+ \frac{b_{k+1}}{2}(1-\sigma^K_{k+1})\frac{(1-\theta_k)\theta_k/\theta_0(1-\theta_0)}{1-\theta_k/\theta_0}\norm{x_k-z_k}_v^2 \\
&+ \big(d_{k+1} - b_{k+1}(1-\sigma^K_{k+1})\frac{\theta_k/\theta_0}{1-\theta_k/\theta_0}\big) \frac 12\norm{x_{k+1}- z_{k+1}}_v^2
  \Big]
\end{align*}
Applying~\eqref{a:Fk} we get:
\begin{align*}
\Exp_k[& C_{k+1}]\leq  (1-\theta_k)(a_{k+1}+\frac{b_{k+1}\sigma^K_{k+1}}{\mu})(\hat F_{k} - F^*) \\
&+ \frac{b_{k+1}}{2}(1-\sigma^K_{k+1})(1-\theta_k)\norm{x_k-x_*}_v^2 \\
&+  \Big(c_{k+1} + b_{k+1}(1-\sigma^K_{k+1})\frac{\theta_k}{\theta_0} -\big(a_{k+1}+\frac{b_{k+1}\sigma^K_{k+1}}{\mu}\big)\frac{\theta_k^2}{\theta_0^2} \Big)\frac 12 \Exp_k[\norm{z_{k+1}-x_*}_v^2]\\
& \Big(\big(a_{k+1}+\frac{b_{k+1}\sigma^K_{k+1}}{\mu}\big)\frac{\theta_k^2}{\theta_0^2}-b_{k+1}(1-\sigma^K_{k+1}) \frac{\theta_k}{\theta_0}(1-\theta_0)\Big)\frac 12 \norm{z_k - x_*}_v^2 \\
&+ \frac{b_{k+1}}{2}(1-\sigma^K_{k+1})\frac{(1-\theta_k)\theta_k/\theta_0(1-\theta_0)}{1-\theta_k/\theta_0}\norm{x_k-z_k}_v^2 \\
&+ \big(d_{k+1} - b_{k+1}(1-\sigma^K_{k+1})\frac{\theta_k/\theta_0}{1-\theta_k/\theta_0}\big) \frac 12\Exp_k[\norm{x_{k+1}- z_{k+1}}_v^2] 
\end{align*}

{\em Step 2: Conditions to get a supermartingale.}

In order to have a bound, we need the following constraints on the parameters for $k \in \{0, \ldots, K-1\}$:
\begin{align}
a_k & \geq (1-\theta_k)(a_{k+1}+\frac{b_{k+1}\sigma^K_{k+1}}{\mu}) \label{a:F} \\
b_k & \geq b_{k+1}(1-\sigma^K_{k+1})(1-\theta_k) \label{b:x} \\
c_{k+1} & \leq \big(a_{k+1}+\frac{b_{k+1}\sigma^K_{k+1}}{\mu}\big)\frac{\theta_k^2}{\theta_0^2} -  b_{k+1}(1-\sigma^K_{k+1})\frac{\theta_k}{\theta_0} \label{c1:z} \\
c_k & \geq \big(a_{k+1}+\frac{b_{k+1}\sigma^K_{k+1}}{\mu}\big)\frac{\theta_k^2}{\theta_0^2}-b_{k+1}(1-\sigma^K_{k+1}) \frac{\theta_k}{\theta_0}(1-\theta_0) \label{c2:z} \\
d_{k+1} & \leq b_{k+1}(1-\sigma^K_{k+1})\frac{\theta_k/\theta_0}{1-\theta_k/\theta_0} \label{d1:xz} \\
d_k & \geq b_{k+1}(1-\sigma^K_{k+1})\frac{(1-\theta_k)\theta_k/\theta_0(1-\theta_0)}{1-\theta_k/\theta_0} \label{d2:xz}
\end{align}
We also add the constraint 
\begin{equation}
\label{a0}
a_0 = (1-\theta_0)(b_0 + c_0) = \frac{1-\theta_0}{\theta_0^2}
\end{equation}
in order to get a bound involving $\Delta(x_0)$.

If we can find a set of sequences satisfying \eqref{a:F}-\eqref{a0}, then we have
\[
a_K (F(x_K) - F^*) + \frac{b_K}{2} \norm{x_{k+1} - x_*}_v^2 \leq \Delta(x_0).
\]

{\em Step 3: Explicit solution of the inequalities.}

By analogy to the gradient case, we are looking for such sequences saturating \eqref{a:F}-\eqref{c2:z}. 
For $k \in \{1,\dots, K-1\}$, equality in \eqref{c1:z} and \eqref{c2:z} can be fulfilled only if
\[
\sigma^K_k = \frac{1 - \frac{\theta_k}{\theta_{k-1}}\frac{1-\theta_0}{1-\theta_k}}{1 + \frac{\theta_{k-1}}{\theta_0 \mu}}, \qquad k \in \{1,\dots, K-1\}
\]
For $k = K$, we would like to ensure $c_K = 0$ and $a_K = (1-\theta_0)b_K$. This can be done by taking
\[
\sigma^K_K = \frac{1 - \frac{\theta_{K-1}}{\theta_0}(1-\theta_0)}{1 + \frac{\theta_{K-1}}{\theta_0\mu}}
\]
Having found $(\sigma^K_k)$, we can get $(b_k)$ as a function of $b_0$ through the equality in~\eqref{b:x} 
\[
b_k = \prod_{l=0}^{k-1} \frac{1}{(1-\theta_l)(1-\sigma^K_{l+1})}b_0 = \frac{\theta_{-1}^2}{\theta_{k-1}^2}\prod_{l=1}^{k} \frac{1}{1-\sigma^K_{l}}b_0 ,
\]
$(a_k)$ as a function of $b_0$ through the equality in~\eqref{a:F}
\[
a_k = \frac{\theta_{-1}^2}{\theta_{k-1}^2} a_0 - \sum_{l=1}^k \frac{\sigma^K_l}{\mu}\frac{\theta_{-1}^2}{\theta_{l-1}^2}b_l.
\]
Using equality in~\eqref{c1:z} and \eqref{c2:z}, we get the relation
\[
c_{k+1} = c_k - \frac{\theta_k}{1-\theta_k}b_k
\]
and so 
\[
c_k = c_0 - \sum_{l=0}^{k-1} \frac{\theta_l}{1-\theta_l}b_l
\]
This gives us the opportunity to calculate $b_0$ because
\[
c_K = 0 = c_0 - \sum_{l=0}^{K-1} \frac{\theta_l}{1-\theta_l}b_l
\]
Hence, we get two equalities:
\begin{align*}
c_0 &+ b_0 = \frac{1}{\theta_0^2} \\
c_0 &= \sum_{l=0}^{K-1} \frac{\theta_l}{1-\theta_l} \frac{\theta_{-1}^2}{\theta_{l-1}^2}\prod_{j=1}^{l} \frac{1}{1-\sigma^K_{j}} b_0
\end{align*}
and so since $\theta_{l-1}^2 = \theta_l^2/(1-\theta_l) $
\[
b_0 = \frac{1}{\theta_0^2} \frac{1}{1 + \sum_{l=0}^{K-1}  \frac{\theta_{-1}^2}{\theta_{l}}\prod_{j=1}^{l} \frac{1}{1-\sigma^K_{j}}}
\]

Finally, we need to check that there is a feasible sequence $(d_k)$. 
Indeed such a sequence exists because the upper and lower bound satisfy 
\begin{align*}
b_{k}(1-\sigma^K_{k})\frac{\theta_{k-1}/\theta_0}{1-\theta_{k-1}/\theta_0} \geq d_k \geq 
b_{k+1}(1-\sigma^K_{k+1})\frac{(1-\theta_k)\theta_k/\theta_0(1-\theta_0)}{1-\theta_k/\theta_0} \\
\overset{eq. in \eqref{b:x}}{\Leftrightarrow} \sigma^K_{k} \leq 1 - (1-\theta_0)\frac{1/\theta_{k-1} - 1/\theta_0}{1/\theta_k - 1/\theta_0}
\end{align*}
which is always true because $\sigma^K_k \leq \theta_0$ for all $k$.

Now that we have the sequence, we can compute the rate as
\begin{align*}
\rho_K & = \frac{b_0 + c_0}{b_K} = \frac{1}{\theta_0^2} \frac{\theta_{K-1}^2}{\theta_{-1}^2}\prod_{l=1}^{K} (1-\sigma^K_{l}) \theta_0^2 \Big(1 + \sum_{l=0}^{K-1} \frac{\theta_{-1}^2}{\theta_{l}}\prod_{j=1}^{l} \frac{1}{1-\sigma^K_{j}}\Big)\\
& =\frac{\theta_{K-1}^2}{\theta_{-1}^2} \Big(\prod_{l=1}^{K} (1-\sigma^K_{l})  + \sum_{l=0}^{K-1} \frac{\theta_{-1}^2}{\theta_{l}}\prod_{j=l+1}^{K} (1-\sigma^K_{j})\Big)
\end{align*}

Hence, we obtain that
\[
\Exp[\Delta(x_K)] \leq
\frac{\theta_{K-1}^2}{\theta_{-1}^2} \Big(\prod_{l=1}^{K} (1-\sigma^K_{l})  + \sum_{l=0}^{K-1} \frac{\theta_{-1}^2}{\theta_{l}}\prod_{j=l+1}^{K} (1-\sigma^K_{j})\Big) \Delta(x_0) = \rho_K \Delta(x_0)
\]
where $\theta_{-1}^2 = \frac{\theta_0^2}{1-\theta_0}$ and
\begin{align*}
\sigma^K_k &= \frac{1 - \frac{\theta_k}{\theta_{k-1}}\frac{1-\theta_0}{1-\theta_k}}{1 + \frac{\theta_{k-1}}{\theta_0 \mu}}, \qquad k \in \{1, K-1\} \\
\sigma^K_K &= \frac{1 - \frac{\theta_{K-1}}{\theta_0}(1-\theta_0)}{1 + \frac{\theta_{K-1}}{\theta_0\mu}}.
\end{align*}

{\em Step 4: bound $\rho_K$ by induction.}
\begin{align*}
\rho_1 &= \frac{\theta_{0}^2}{\theta_{-1}^2} \Big((1-\sigma_{1}^1)  +  \frac{\theta_{-1}^2}{\theta_{0}}(1-\sigma_{1}^1)\Big) = (1-\theta_0)\frac{\frac{\theta_0}{\theta_0\mu} + \frac{\theta_0}{\theta_0}(1-\theta_0)}{1 + \frac{\theta_0}{\theta_0 \mu}}(1+\frac{\theta_0}{1-\theta_0}) \\
&= \frac{1 + (1-\theta_0)\mu}{1+\mu} \leq \frac{1 + (1-\theta_0)\mu}{1+\frac{\theta_0^2}{2\theta_0^2}\mu}
\end{align*}
Let us now assume that $\rho_K \leq \frac{1 + (1-\theta_0)\mu}{1+\frac{\theta_0^2\mu}{2\theta_{K-1}^2}}$.
\begin{align*}
\rho_{K+1} &= \frac{\theta_{K}^2}{\theta_{-1}^2} \Big(\prod_{l=1}^{K+1} (1-\sigma^{K+1}_{l})  + \sum_{l=0}^{K} \frac{\theta_{-1}^2}{\theta_{l}}\prod_{j=l+1}^{K+1} (1-\sigma^{K+1}_{j})\Big) \\
&= \frac{\theta_{K}^2}{\theta_{K-1}^2}\frac{\theta_{K-1}^2}{\theta_{-1}^2}(1-\sigma_{K+1}^{K+1})\frac{1-\sigma_{K}^{K+1}}{1-\sigma_K^K} \Big(\prod_{l=1}^{K} (1-\sigma^{K}_{l})  + \sum_{l=0}^{K-1} \frac{\theta_{-1}^2}{\theta_{l}}\prod_{j=l+1}^{K} (1-\sigma^{K}_{j}) \\
& \qquad \qquad+ \frac{\theta_{-1}^2}{\theta_{K}} \frac{1-\sigma^{K}_{K}}{1-\sigma_K^{K+1}} \Big) \\
& = (1-\theta_K) (1-\sigma_{K+1}^{K+1})\frac{1-\sigma_{K}^{K+1}}{1-\sigma_K^K} \Big(\rho_K +  \frac{\theta_{K-1}^2}{\theta_K}\frac{1-\sigma^{K}_{K}}{1-\sigma_K^{K+1}} \Big) \\
& = (1-\theta_K) \frac{1+ \frac{\theta_0 \mu}{\theta_{K}}(1-\theta_0)}{1+\frac{\theta_0 \mu}{\theta_{K}}} \rho_K + \frac{1 + (1-\theta_0) \mu}{1+ \frac{\theta_0 \mu}{\theta_K}}\theta_K \\
\end{align*}
We now use the induction hypothesis.
\begin{align*}
\rho_{K+1} & \leq (1-\theta_K) \frac{1+ \frac{\theta_0 \mu}{\theta_{K}}(1-\theta_0)}{1+\frac{\theta_0 \mu}{\theta_{K}}}  \frac{1 + (1-\theta_0)\mu}{1+\frac{\theta_0^2}{2\theta_{K-1}^2}\mu}+ \frac{1 + (1-\theta_0) \mu}{1+ \frac{\theta_0 \mu}{\theta_K}}\theta_K \\
& \leq \frac{1 + (1-\theta_0)\mu}{1+\frac{\theta_0^2\mu}{2\theta_{K}^2}}
\end{align*}
The last inequality comes from the fact that 
\begin{align*}
 &\frac{1+ \frac{\theta_0 \mu}{\theta_{K}}(1-\theta_0)}{1+\frac{\theta_0 \mu}{\theta_{K}}}  \frac{1-\theta_K}{1+\frac{\theta_0^2\mu}{2\theta_{K-1}^2}}+ \frac{\theta_K}{1+ \frac{\theta_0 \mu}{\theta_K}} \leq \frac{1}{1+\frac{\theta_0^2\mu}{2\theta_{K}^2}} \\
 & \Leftrightarrow (1-\theta_K) (1+ \frac{\theta_0 \mu}{\theta_{K}}(1-\theta_0)) (1+\frac{\theta_0^2\mu}{2\theta_{K}^2}) + \theta_K (1+\frac{\theta_0^2\mu}{2\theta_{K}^2})(1+\frac{\theta_0^2\mu}{2\theta_{K-1}^2}) \\
 & \qquad \qquad\leq (1+\frac{\theta_0^2\mu}{2\theta_{K-1}^2})(1+\frac{\theta_0 \mu}{\theta_{K}}) \\
 & \Leftrightarrow 0 + \mu\Big[ (1-\theta_K)\frac{\theta_0(1-\theta_0)}{\theta_{K}} + \frac{(1-\theta_K)\theta_0^2}{2\theta_{K}^2} + \frac{\theta_K\theta_0^2}{2\theta_{K-1}^2} + \frac{\theta_K\theta_0^2}{2 \theta_{K}^2} - \frac{\theta_0^2}{2 \theta_{K-1}^2} - \frac{\theta_0}{\theta_K}\Big] \\
 & \qquad \qquad + \mu^2 \Big[ \frac{(1-\theta_K)\theta_0^3(1-\theta_0)}{2 \theta_K^2 \theta_{K}} + \frac{\theta_K\theta_0^4}{4 \theta_{K-1}^2 \theta_K^2} - \frac{\theta_0^3}{2 \theta_{K-1}^2\theta_K} \Big] \leq 0
\end{align*}

and both terms in the brackets are nonpositive for all $K$.
\end{proof}

%

\bibliographystyle{alpha}
\bibliography{../../literature}

\end{document}